\documentclass[12pt]{amsart}
\usepackage{amsmath,latexsym,amscd,amsbsy,amssymb,amsfonts,amsthm,fleqn,leqno,
euscript, graphicx, texdraw}

\numberwithin{equation}{section}

\newtheorem{thm}{Theorem}[section]
\newtheorem{pro}[thm]{Proposition}
\newtheorem{lem}[thm]{Lemma}
\newtheorem{cor}[thm]{Corollary}

\theoremstyle{definition}

\theoremstyle{remark}

\begin{document}

\title[Locally $n$-connected compacta and $UV^n$-maps]
{Locally $n$-connected compacta and $UV^n$-maps}

\author{V. Valov}
\address{Department of Computer Science and Mathematics,
Nipissing University, 100 College Drive, P.O. Box 5002, North Bay,
ON, P1B 8L7, Canada} \email{veskov@nipissingu.ca}

\date{\today}
\thanks{The author was partially supported by NSERC
Grant 261914-13.}


\begin{abstract}
We provide a machinery for transferring some properties of metrizable $ANR$-spaces to metrizable $LC^n$-spaces. As a result, we show that for complete metrizable spaces the properties $ALC^n$, $LC^n$ and $WLC^n$ coincide to each other. We also provide the following spectral characterizations of  $ALC^n$ and cell-like compacta:
A compactum $X$ is $ALC^n$ if and only if $X$ is the limit space of a $\sigma$-complete inverse system $\displaystyle S=\{X_\alpha, p^{\beta}_\alpha, \alpha<\beta<\tau\}$ consisting of compact metrizable $LC^n$-spaces $X_\alpha$ such that all bonding projections $p^{\beta}_\alpha$, as a well all limit projections $p_\alpha$, are $UV^n$-maps.\\
A compactum $X$ is a cell-like $($resp., $UV^n$$)$ space if and only if $X$ is the limit space of a $\sigma$-complete inverse system consisting of
cell-like $($resp., $UV^n$$)$ metric compacta.
\end{abstract}
\maketitle\markboth{}{Locally $n$-connected compacta}





\section{Introduction}

Following \cite{dr}, we say that a space $M$ is weakly locally $n$-connected (briefly, $WLC^n$) in a space $Y$ if $M\subset Y$ is closed and for every point $x\in M$ and its open neighborhood $U$ in $M$ there exists a neighborhood $V$ of $x$ in $M$ such that any map from $\mathbb S^k$, $k\leq n$, into $V$ is null-homotopic in $\widetilde U$, where $\widetilde U$ is any open in $Y$ set with $\widetilde U\cap M=U$ (in such a case we say that $\widetilde U$ is an open extension of $U$ in $Y$).
Dranishnikov \cite{dr2} also suggested the following notion: a space $M$ is approximately locally $n$-connected (briefly, $ALC^n$) in a space $Y$ if $M\subset Y$ is closed and for every point $x\in M$ and its open neighborhood $U$ in $M$ there exists a neighborhood $V$ of $x$ in $M$ such that for any open in $Y$ extension $\widetilde U$ of $U$ there exists and open in $Y$ extension $\widetilde V$ of $V$ with any map from $\mathbb S^k$, $k\leq n$, into $\widetilde V$ being null-homotopic in $\widetilde U$. One can show that if $M$ is metrizable (resp., compact), then $M$ is $WLC^n$ in a given space $Y$, where $Y$ is a metrizable (resp., compact) $ANR$, if and only if $M$ is $WLC^n$ in any metrizable (resp., compact) $ANR$ containing $M$ as a closed set. The same is true for the property $ALC^n$. So, the definitions of $WLC^n$ and $ALC^n$ don't depend on the $ANR$-space containing $M$, and we say that $M$ is $WLC^n$ (resp., $ALC^n$).

Dranishnikov \cite{dr} proved that both properties $WLC^n$ and $LC^n$ are identical in the class of metrizable compacta. It also follows from Gutev \cite{gu} and Dugundji-Michael \cite{dm} that this remains true for complete metrizable spaces. One of the main result in this paper, Theorem 2.7, shows that all properties $WLC^n$,  $LC^n$ and $ALC^n$ coincide for completely metrizable spaces. The proof of Theorem 2.7 is based on the technique, developed in Section 2, for transferring properties of metrizable $ANR$'s to $LC^n$-subspaces (in this way well known properties of metrizable $LC^n$-spaces can be obtained from the corresponding properties of metrizable $ANR$'s, see for example Proposition 2.2). Section 2 contains also a characterization of metrizable $LC^n$-spaces whose analogue for $ANR$'s was established by Nhu \cite{nhu}.

It is well known that the class of metrizable $LC^{n}$-spaces are exactly absolute neighborhood extensors for $(n+1)$-dimensional paracompact spaces, and this is not valid for non-metrizable spaces. Outside the class of metrizable spaces we have the following characterization of $ALC^n$ compacta (Theorem 3.1):
A compactum $X$ is $ALC^n$ if and only if $X$ is the limit space of a $\sigma$-complete inverse system $\displaystyle S=\{X_\alpha, p^{\beta}_\alpha, \alpha<\beta<\tau\}$ consisting of compact metrizable $LC^n$-spaces $X_\alpha$ such that all bonding projections $p^{\beta}_\alpha$, as a well all limit projections $p_\alpha$, are $UV^n$-maps. A similar spectral characterization is obtained for cell-like or $UV^n$ compacta, see Theorem 3.3. Both Theorem 3.1 and Theorem 3.3 provide different classes of compacta $\mathcal C$ and corresponding classes of maps $\mathcal M$ adequate with $\mathcal C$ in the following sense (see Shchepin \cite{s76}): A compactum $X$ belongs to $\mathcal C$ if and only if $X$ is the limit space of a $\sigma$-complete inverse system $\displaystyle S=\{X_\alpha, p^{\beta}_\alpha, \alpha<\beta<\tau\}$ consisting of compact metrizable $X_\alpha\in\mathcal C$ with all bonding projections $p^{\beta}_\alpha$ being from $\mathcal M$. For example, according to Theorem 3.1, the class of $ALC^n$ compacta is adequate with the class of $UV^n$-maps.

Recall that a closed subset $A\subset X$ is $UV^n$ in $X$ (resp., cell-like in $X$) if  every neighborhood $U$ of $A$ in $X$ contains a neighborhood $V$ of $A$ such that, for each $0\leq k\leq n$, any map $f\colon\mathbb S^k\to V$ is null-homotopic in $U$ (resp., $A$ is contractible in every neighborhood of $A$ in $X$). A space $X$ is said to be $UV^n$ (resp., cell-like) provided it is $UV^n$ (resp., cell-like) in some $ANR$-space containing $X$ as a closed set.
It is well known that  for metrizable or compact $X$ this definition does not depend on the $ANR$-spaces containing $X$ as a closed set, see for example \cite{sa}. A map $f\colon X\to Y$ between compact spaces is called $UV^n$ (resp., cell-like) if all fibres of $f$ are $UV^n$ (resp., cell-like).

Theorem 3.1 yields that any compact $LC^n$-space is $ALC^n$. It is interesting to find an example of an $ALC^n$-compactum which is not $LC^n$ (obviously, such a compactum should be non-metrizable).

\section{Metrizable $ALC^n$-spaces}

We are going to establish some properties of metric $LC^n$-spaces using the corresponding properties of metric $ANR$-spaces. Recall that a map $f\colon X\to Y$ is {\em $n$-invertible} if for every Tychonoff space $Z$ with $\dim Z\leq n$ and any map $g\colon Z\to Y$ there is a map $h\colon Z\to X$ such that $g=f\circ h$.

The next theorem follows from a stronger result due to Pasynkov \cite[Theorem 6]{pa}, and its proof is based on Dranishnikov results \cite[Theorem 1]{dr1} and
\cite[Theorem 1.2]{dr} (such a result concerning the extension dimension with respect to quasi-finite complexes  was established in \cite[Proposition 2.7]{kv}). We provide here a proof based on factorization theorems.

\begin{thm}
 Every Tychonoff space $M$ is the image of a Tychonoff  space $X$ with $\dim X\leq n$ under a perfect $n$-invertible map. In case $M$ is metrizable, $X$ can be supposed to be also a metrizable space with $w(X)=w(M)$.
\end{thm}

\begin{proof}
Let $M$ be a Tychonoff space of weight $\tau$. Consider all couples $(Z_\alpha,f_\alpha)$, where $Z_\alpha$ is a Tychonoff space of weight $w(Z_\alpha)\leq w(\beta M)$, $\dim Z_\alpha\leq n$ and
$f_\alpha$ is a map from $Z_\alpha$ into $M$ (here $\beta M$ is the \v{C}ech-Stone compactification of $M$).
Denote by $Z$ the disjoint sum of all spaces $Z_\alpha$. Obviously, there is a natural map $f\colon Z\to M$
such that $f|Z_\alpha=f_\alpha$ for all $\alpha$. Let $\widetilde{f}\colon\beta Z\to\beta M$ be the extension of $f$. Then, by the Marde\v{s}i\'{c}'s factorization theorem \cite{m}, there exists a compactum
$\widetilde X$ of weight $w(\widetilde X)\leq w(\beta M)$ with $\dim\widetilde{X}\leq\dim\beta Z=n$ and maps $h\colon \beta Z\to\widetilde X$,  $\widetilde{g}\colon\widetilde X\to\beta M$ such that $\widetilde{g}\circ h=\widetilde{f}$. Let $X=\widetilde{g}^{-1}(M)$ and $g=\widetilde{g}|X$. According to Corollary 6 and Main Theorem from \cite{pa}, $\dim X\leq n$. To show that $g$ is $n$-invertible, suppose $f_0\colon Z_0\to M$ is a map with $\dim Z_0\leq n$.
Applying again the Marde\v{s}i\'{c}'s factorization theorem for the map $\widetilde{f_0}\colon\beta Z_0\to\beta M$, we obtain a compactum $K$ and maps
$h_1\colon\beta Z_0\to K$ and $f_2\colon K\to\beta M$ such that $\dim K\leq n$, $w(K)\leq w(\beta M)$ and $f_2\circ h_1=\widetilde{f_0}$. Then, as above,  $Z'=f_2^{-1}(M)$ is a space of dimension $\leq n$ and weight $\leq w(\beta M)$. So, there exists $\alpha^*$ and a homeomorphism $j\colon Z'\to Z$ such that  $j(Z')=Z_{\alpha^*}$ and $f_{\alpha^*}\circ j=f_2|Z'$. Consequently, $h\circ j$ is a map from $Z'$ to $X$ with $f_2|Z'=g\circ h\circ j$. Finally, $h\circ j\circ h_1$ is a map from $Z_0$ to $X$ such that $g\circ h\circ j\circ h_1=f_0$.

If $M$ is metrizable, the proof is simpler. Indeed, in this case $P=\widetilde{f}^{-1}(M)$ is a space of dimension $\leq n$ and the restriction
$\widetilde{f}|P$ is a perfect map. So, by Pasynkov's factorization theorem \cite{pa1}, there exists a metrizable space $X$ and maps $h\colon P\to X$,
$g\colon X\to M$ such that $g\circ h=\widetilde{f}$, $w(X)\leq w(M)$ and $\dim X\leq n$. Then $g$ is a perfect map because so is $\widetilde{f}|P$, and according to the above arguments, $g$ is $n$-invertible.
\end{proof}

 The next proposition shows that Theorem 2.1 allows some properties of metrizable $ANR$-spaces to be transferred to metrizable $LC^n$-spaces.
\begin{pro}
Let $M$ be a metrizable $LC^n$-space and $\alpha$ an open cover of $M$. Then there exists an open cover $\beta$ of $M$ refining $\alpha$ such that for any two $\beta$-near maps $f,g\colon Z\to M$ defined on a metrizable space $Z$ of dimension $\leq n$ any $\beta$-homotopy $H\colon A\times [0,1]\to M$ between $f|A$ and $g|A$, where $A$ is closed in $Z$, can be extended to  an $\alpha$-homotopy $\widetilde{H}\colon Z\times [0,1]\to M$ connecting $f$ and $g$.
\end{pro}

\begin{proof}
We embed $M$ as a closed subset of a metrizable $ANR$-space $P$ and let $p\colon Y_P\to P$ be a perfect $(n+1)$-invertible surjection such that $Y_P$ is a metrizable space of dimension $\leq n+1$ (see Theorem 2.1). Since $M$ is $LC^n$, there is an open set $G$ in $Y_P$ containing $p^{-1}(M)$ and a map $q\colon G\to M$ extending the restriction $p|p^{-1}(M)$. Then there exists an open set $W\subset P$ containing $M$ with $p^{-1}(W)\subset G$ (recall that $p$ is a perfect map). Obviously $W$ is also an $ANR$-space containing $M$ as a closed set. So, without losing generality, we may assume that $W=P$, $G=Y_P$ and $q$ is a map from $Y_P$ onto $M$. Now, for every open $U\subsetneqq M$, let $\widetilde{U}=P\backslash(p(q^{-1}(M\backslash U)))$. The set $\widetilde{U}$ is non-empty and open in $P$, $\widetilde{U}\cap M=U$ and $p^{-1}(\widetilde{U})\subset q^{-1}(U)$.

If  $\alpha$ is an open cover of $M$ consisting of proper subsets of $M$, the set $L=\bigcup\{\widetilde{U}:U\in\alpha\}$ is open in $P$ and contains $M$. So, $L$ is also an $ANR$ and $\widetilde{\alpha}=\{\widetilde{U}:U\in\alpha\}$ is an open cover of $L$. According to the properties of metrizable $ANR$'s (see for example \cite[chapter IV, Theorem 1.2]{hu}), there exists an open cover
$\widetilde{\beta}$ of $L$ with the following property: for any two $\beta$-near maps $h_1,h_2\colon Z\to M$ defined on a metrizable space $Z$ any $\beta$-homotopy $H\colon A\times [0,1]\to M$ between $h_1|A$ and $h_2|A$, where $A$ is closed in $Z$, can be extended to  an $\alpha$-homotopy $F\colon Z\times [0,1]$ connecting $h_1$ and $h_2$. Then $\beta=\{V\cap M:V\in\widetilde{\beta}\}$ is an open cover of $M$ refining $\alpha$ and has the desired property. Indeed, suppose $f,g\colon Z\to M$ are two $\beta$-near maps and $H\colon A\times [0,1]\to M$ is a $\beta$-homotopy between $f|A$ and $g|A$, where $Z$ is a metrizable space of dimension $\leq n$ and $A$ is closed in $Z$. According to the choice of $\widetilde{\beta}$, $H$ can be extended to an $\widetilde{\alpha}$-homotopy $F\colon Z\times [0,1]\to L$ connecting $f$ and $g$. Since $\dim Z\times [0,1]\leq n+1$ and $p$ is $(n+1)$-invertible, there exists a lifting $F_1\colon Z\times [0,1]\to p^{-1}(L)$ of $F$. Finally, $\widetilde{H}=q\circ F_1$ is an $\alpha$-homotopy between $f$ and $g$.
\end{proof}

We also need the following property of metrizable $LC^n$-spaces.
\begin{pro}
Suppose both $M$ and $P$ are metrizable $LC^n$-spaces with $M\subset P$ being closed. Then there exists an open set $U\subset P$ containing $M$ with the following property: If $Z$ is a metrizable space with $\dim Z\leq n$ and $h\colon Z\to U$ is a map, there exists a map $h_1\colon Z\to M$ such that $h$ and $h_1$ are homotopic in $P$.
\end{pro}

\begin{proof}
Let $p\colon Y_P\to P$ be a perfect $(n+1)$-invertible surjection such that $Y_P$ is a metrizable space of dimension $\leq n+1$, and $q\colon p^{-1}(W)\to M$ extends the restriction $p|p^{-1}(W)$, where $W\subset P$ is an open set containing $M$ (see the proof of Proposition 2.2). Since $W$ (as an open subset of $P$) is $LC^n$, we can apply Proposition 2.2 (with $\alpha=\{W\}$ and $A$ a point), to obtain that $W$ has an open cower $\beta$ such that any two $\beta$-near maps from $Z$ into $W$ are homotopic. For every $V\in\beta$ let
$G_V=V\backslash p(q^{-1}(M\backslash (M\cap V)))$ and $U=\bigcup\{G_V:V\in\beta\}$. Because $\beta$ covers $M$, $M\subset U$. If $h\colon Z\to U$ is any map, let $h_1\colon Z\to M$ be the map $h_1=q\circ\widetilde h$, where $\widetilde h\colon Z\to p^{-1}(U)$ is a lifting of $h$. It is easily seen that the maps $h$ and $h_1$ are $\beta$-near. So, $h$ and $h_1$ are homotopic in $W$, and hence in $P$.
\end{proof}

\begin{cor}
If $M$ and $P$ are metrizable $LC^n$-spaces such that $M\subset P$ is closed, then $M$ is $ALC^n$ in $P$.
\end{cor}

\begin{proof}
Since $M$ is $LC^n$, for every $x\in M$ and its open neighborhood $U$ in $M$ there exists a neighborhood $V$ of $x$ in $M$ such that any map of a $k$-sphere, $k\leq n$, into $V$ is contractible in $U$. Let $\widetilde U\subset P$ and $G\subset\widetilde U$ be open in $P$ extensions of $U$ and $V$, respectively. Then both $V$ and $G$ are $LC^n$ (as open subsets of $LC^n$-spaces), and $V$ is closed in $G$. So, there is an open extension $\widetilde V\subset G$
satisfying the conclusion from Proposition 2.4. Consequently, any map $g\colon\mathbb S^k\to\widetilde V$, $0\leq k\leq n$, is homotopic in $G$ to a map
$g_1\colon\mathbb S^k\to V$. Since $g_1$ is homotopic in $U$ to a constant map, we obtain that $g$ is homotopic in $\widetilde U$ to a constant map.
\end{proof}

\begin{pro}
Let $P$ be metrizable and $M\subset P$ be a closed and $LC^n$-set. Then every closed set $A\subset M$ is $UV^n$ in $M$ provided $A$ is $UV^n$ in $P$.
\end{pro}

\begin{proof}
Let $p\colon Y_P\to P$ be a perfect $(n+1)$-invertible surjection such that $Y_P$ is a metrizable space of dimension $\leq n+1$, and $q\colon p^{-1}(W)\to M$ extends the restriction $p|p^{-1}(W)$, where $W\subset P$ is an open set containing $M$. Suppose $U\subset M$ is an open set containing $A$ and let $\widetilde U=W\backslash p(q^{-1}(M\backslash U))$. Obviously, $\widetilde U\subset W$ is an open extension of $U$. Since $A\in UV^n(P)$, there is an open set $\widetilde V\subset\widetilde U$ containing $A$ such that any map from  $\mathbb S^k$ with $0\leq k\leq n$ to $\widetilde V$ can be extended to a map from $\mathbb B^{k+1}$ to $\widetilde U$. Let $V=M\cap\widetilde V$ and $g\colon\mathbb S^k\to V$ be a map. Then extend $g$  to a map $g_1\colon\mathbb B^{k+1}\to\widetilde U$ and take a lifting $g_2\colon\mathbb B^{k+1}\to p^{-1}(\widetilde U)$ of $g_1$. Finally,
$\widetilde g=q\circ g_2$ is a map from $\mathbb B^{k+1}$ to $U$ extending $g$.
\end{proof}

Next lemma is a non-compact analogue of Lemma 2.1 from \cite{dr}.
\begin{lem}
Suppose both $M$ and $P$ are metric $LC^{n-1}$-spaces such that $M$ is a closed subset of $P$. Then for every $\epsilon>0$ there exists a neighborhood $U_\epsilon(M)$ in $P$ such that for any map $\varphi\colon (Q,Q_0)\to (U_\epsilon(M),M)$, where $(Q,Q_0)$ is a polyhedral pair, there exists a map $\psi\colon (Q^{(n)},Q_0)\to M$ such that $\psi|Q_0=\varphi|Q_0$ and both $\varphi$ and $\psi$ are $\epsilon$-close.
\end{lem}

\begin{proof}
The proof is similar to that one of Proposition 2.5, the only difference is that the map $p:Y_P\to P$ is $n$-invertible and $\dim Y_P\leq n$. We take an open cover $\omega$ of $M$ with each $V\in\omega$ having a diameter $<\epsilon$. For every
$V\in\omega$ consider the set $\widetilde V=P\backslash p(q^{-1}(M\backslash V))$ and let $U_\epsilon(M)=\bigcup\{\widetilde V:V\in\omega\}$. If $\varphi\colon (Q,Q_0)\to (U_\epsilon(M),M)$, we first lift $\varphi|Q^{(n)}$ to a map $\varphi_1\colon Q^{(n)}\to Y_P$ and define $\psi\colon (Q^{(n)},Q_0)\to M$ to be the map
with $\psi|Q_0=\varphi|Q_0$ and $\psi|Q^{(n)}=q\circ\varphi_1$. Obviously, $\psi$ satisfies the required conditions.
\end{proof}

Now, we are in a position to prove the main theorem in this section.

\begin{thm}
For a complete metric space $(M,d)$ the following are equivalent:
\begin{itemize}
\item[(i)] $M$ is $LC^n$;
\item[(ii)] $M$ is $ALC^n$;
\item[(iii)] $M$ is $WLC^n$.
\end{itemize}
\end{thm}

\begin{proof}
Implication $(i)\Rightarrow (ii)$ follows from Corollary 2.4, and implication $(ii)\Rightarrow (iii)$ is trivial.\\
$(iii)\Rightarrow (i)$. We are going to prove this implication by induction. Since $M$ is $LC^{-1}$ (there is no such thing as a $(-1)$-sphere), we can suppose that $M$ is $LC^{n-1}$ and $WLC^n$.
We embed $M$ as a closed subset of a complete metric $ANR$-space $P$ and consider the following relation between the open subsets of $M$: $V\alpha U$ if $V\subset U$ and every map from $\mathbb S^k$, $k\leq n$, into $V$ is null-homotopic in $\widetilde U$, where $\widetilde U$ is any open extension of $U$ in $P$. It follows from \cite[Theorem 1]{dm} the existence a complete metric $\rho$ on $M$ generating its topology such that for every $\epsilon>0$ there exists $\delta(\epsilon)>0$ with $B^\rho_{\delta(\epsilon)}(x)\alpha B^\rho_\epsilon(x)$ for all $x\in M$, where
$B^\rho_\epsilon(x)=\{y\in M:\rho(x,y)<\epsilon\}$. The metric $\rho$ can be extended to a complete metric $\widetilde\rho$ on $P$, see \cite{ba}.
We fix a sequence $\{G_k\}_{k\geq 1}$ of open subsets of $P$ containing $M$ with $\bigcap_{k=1}^\infty G_k=M$, and define by induction a decreasing sequence $\{W_k\}_{k\geq 1}$ of open subsets of $P$ such that
$W_k\subset U_{2^{-k}}(M)\cap G_k$, where $U_{2^{-k}}(M)$ is a neighborhood of $M$ in $P$ corresponding to $2^{-k}$ (see Lemma 2.6). Let $\{g_k\}_{k\geq 1}$ be a sequence of continuous functions $g_k\colon P\to [0,1]$ such that $g_k(M)=1$ and $g_k(P\backslash W_k)=0$. Then the equality
$\displaystyle\varrho(x,y)=\widetilde\rho(x,y)+\sum_{k=1}^\infty\frac{|g_k(x)-g_k(y)|}{2^{k}}$  provides a metric on $P$. It easily seen that $\varrho$ is a complete metric generating the topology of $P$. Moreover, $\varrho(x,y)=\widetilde\rho(x,y)=\rho(x,y)$ for all $x,y\in M$. For any $x\in M$ and $y\in P\backslash W_k$ we have $$\displaystyle\varrho(x,y)=\widetilde\rho(x,y)+\sum_{j=1}^{j=k-1}\frac{|1-g_j(y)|}{2^{j}}+\sum_{j=k}^\infty\frac{1}{2^{j}}\geq\frac{1}{2^{k-1}}.$$
Hence, $\varrho(M,P\backslash W_k)\geq 2^{-k+1}$ for all $k$. For every $\epsilon>0$ let $k_\epsilon=\min\{k:2^{-k}<\epsilon\}$,
$\eta(\epsilon)=2^{-k_\epsilon+1}$ and $B_\epsilon^\varrho(M)=\{y\in P:\varrho(y,M)<\epsilon\}$.

\smallskip
\textit{Claim $1$. Let $(Q,Q_0)$ be a polyhedral pair and $\epsilon>0$. Then for any map $\varphi: Q\to B_{\eta(\epsilon)}^\varrho(M)$ with $\varphi(Q_0)\subset M$ there exists a map $\psi:Q^{(n)}\cup Q_0\to M$ such that $\psi|Q_0=\varphi|Q_0$ and $\psi$ is $\epsilon$-close to $\varphi$}.

Since
$\varrho(M,P\backslash W_{k_\epsilon})\geq 2^{-k_\epsilon+1}=\eta(\epsilon)$, $B_{\eta(\epsilon)}^\varrho(M)\subset W_{k_\epsilon}$. Hence,
$B_{\eta(\epsilon)}^\varrho(M)\subset U_{2^{-k_\epsilon}}$, and according to Lemma 2.6, there exists a map such that $\psi:Q^{(n)}\cup Q_0\to M$ such that $\psi|Q_0=\varphi|Q_0$ and $\psi$ is $(2^{-k_\epsilon})$-close to $\varphi$. Then the inequality $2^{-k_\epsilon}<\epsilon$ completes the proof of Claim 1.

To prove that $M$ is $LC^n$, we introduce the following notation: If $U\subset M$ is open and $\gamma>0$, then $E_\gamma(U)$ denotes any open extension of $U$ in $P$ witch is contained in the set $B^\varrho_\gamma(M)$. Since $M$ is $WLC^n$, according to the choice of the function $\delta$, any map $\varphi:\mathbb S^n\to M$ with
$\mathrm{diam}\varphi(\mathbb S^n)<\delta(\epsilon)$ can be extended to a map $\widetilde{\varphi}\colon\mathbb B^{n+1}\to E_\gamma(B_\epsilon^\rho(\varphi(\mathbb S^n))$, where $\gamma>0$ is arbitrary.
Now, we
proceed as in the proof of Lemma 2.5 from \cite{dr}. Fix $\epsilon>0$ and a map $f\colon\mathbb S^n\to M$ with $\mathrm{diam} f(\mathbb S^n)<\delta(\epsilon/2)$. We are going to show that $f$ can be extended to a map $\widetilde f\colon\mathbb B^{n+1}\to M$ such that
$\mathrm{diam} f(\mathbb B^{n+1})\leq 10\epsilon$.
The map $\widetilde f$ will be obtained as limit of a sequence $\{f_k\colon\mathbb B^{n+1}\to P\}$, and this sequence will be constructed by induction together with a sequence of triangulations $\{\tau_k\}$ of $\mathbb B^{n+1}$ such that $f_k(\tau_k^{(n)})\subset M$ for all $k$. To start the induction, we choose $\tau_1$ to be $\mathbb B^{n+1}$, considered as one $(n+1)$-dimensional simplex, and let $\displaystyle f_1\colon\mathbb B^{n+1}\to E_{\eta(\delta(2^{-1}\epsilon)/3)}\big(B^\varrho_{\epsilon}(z)\big)$ be a map extending $f$, where $z\in f(\mathbb S^{n})$. Suppose a triangulation $\tau_k$ of $\mathbb B^{n+1}$ and a map $f_k\colon\mathbb B^{n+1}\to B^\varrho_{\eta(\delta(2^{-k}\epsilon)/3)}(M)$ are already constructed with $f_k(\tau_k^{(n)})\subset M$. Denote by $\tau_{k+1}$ any subdivision of  $\tau_k$ such that $\mathrm{diam} f_k(\sigma)<\delta(2^{-k}\epsilon)/3$ for any simplex $\sigma\in\tau_{k+1}$. According to Claim 1, there exists a map $g_k\colon\tau_{k+1}^{(n)}\to M$ such that $g_k$ is $(\delta(2^{-k}\epsilon)/3)$-close to the restriction $f_k|\tau_{k+1}^{(n)}$ and $g_k|\tau_k^{(n)}=f_k|\tau_k^{(n)}$. Therefore, $\mathrm{diam} g_k(\partial\sigma)<\delta(2^{-k}\epsilon)$ for any $\sigma\in\tau_{k+1}$. Hence, $g_k|\partial\sigma$ can be extended to a map
$g_k^\sigma\colon\sigma\to E_{\eta(\delta(2^{-k-1}\epsilon)/3)}(B_{\epsilon/2^k}^\rho(g_k(\partial\sigma)))$, $\sigma\in\tau_{k+1}$. Finally, we define $f_{k+1}:\mathbb B^{n+1}\to B^\varrho_{\eta(\delta(2^{-k-1}\epsilon)/3)}(M)$, $f_{k+1}|\sigma=g_k^\sigma$. Because $\delta(\epsilon)\leq\epsilon$ and
$\eta(\epsilon)\leq\epsilon$, for every $\sigma\in\tau_{k+1}$ we have
 $\mathrm{diam} f_{k+1}(\sigma)\leq\mathrm{diam} g_k(\partial\sigma)+\epsilon/2^k+\eta(\delta(2^{-k-1}\epsilon)/3)\leq\delta(2^{-k}\epsilon)+\epsilon/2^k+\eta(\delta(2^{-k-1}\epsilon)/3)\leq 3\epsilon/2^k$. Now, for every $x\in\mathbb B^{n+1}$ take $\sigma\in\tau_{k+1}$ containing $x$ and $y\in\partial\sigma$. Then $\varrho(f_k(x),f_{k+1}(x))\leq\varrho(f_{k+1}(x),f_{k+1}(y))+\varrho(f_k(y),f_{k+1}(y))+\varrho(f_k(y),f_{k}(x))$.
 Consequently, $\varrho(f_k(x),f_{k+1}(x))\leq 3\epsilon/2^{k}+2\delta(2^{-k}\epsilon)/3\leq 4\epsilon/2^k$ for all $x\in\mathbb B^{n+1}$. The last inequality implies
 that the sequence $\{f_k(x)\}$ is convergent for all $x\in M$. So, the limit map $\widetilde f\colon\mathbb B^{n+1}\to P$ is well defined and $\varrho(f_1(x),\widetilde f(x))\leq 4\epsilon$. Then for every $x,y\in\mathbb B^{n+1}$ we have
 $\varrho(\widetilde f(x),\widetilde f(y))\leq\varrho(\widetilde f(x),f_1(x))+\varrho(f_1(x),f_1(y))+\varrho(f_1(y),\widetilde f(y))\leq 4\epsilon+\epsilon+\eta(\delta(\epsilon/2)/3)+4\epsilon\leq 10\epsilon$.
 Since $f_k(\mathbb B^{n+1})\subset B^\varrho_{\eta(\delta(2^{-k}\epsilon)/3)}(M)$ and $\lim \eta(\delta(2^{-k}\epsilon)/3)=0$, $\widetilde f(\mathbb B^{n+1})\subset M$. This completes the proof.
\end{proof}

A sequence of open covers $\mathcal U=(\mathcal U_k)_{k\in\mathbb N}$  of a metric space $(M,d)$ is called a {\em zero-sequence} if $\lim_{k\mapsto\infty}\rm{mesh} \mathcal{U}_k=0$.
For any such a sequence we define $\rm{Tel}(\mathcal U)=\bigcup_{k\in\mathbb N} N(\mathcal{U}_k\cup\mathcal{U}_{k+1})$. Here $N(\mathcal{U}_k\cup\mathcal{U}_{k+1})$ is the nerve of $\mathcal{U}_k\cup\mathcal{U}_{k+1}$ with $\mathcal{U}_{k}$ and $\mathcal{U}_{k+1}$ considered as disjoint sets. For any $\sigma\in\rm{Tel}(\mathcal U)$ let $s(\sigma)=\max\{s:\sigma\in N(\mathcal{U}_s\cup\mathcal{U}_{s+1})\}$.

We complete this section by a characterization of metrizable $LC^n$-spaces similar to the characterization of metrizable $ANR$-spaces  provided in \cite{nhu} (see also \cite[Theorem 6.8.1]{sa}).
\begin{pro}
A metric space $(M,d)$ is $LC^n$ if and only if it has a zero-sequence $\mathcal U$ of open covers such that any map $f_0\colon K^{(0)}\to M$ with $f(U)\in U$, $U\in K^{(0)}$, where $K$ is a subcomplex of $\rm{Tel}(\mathcal U)$, extends to a map $f\colon K^{(n+1)}\to M$ satisfying the following condition:
\begin{itemize}
\item[(*)] For any sequence $\{\sigma_k\}$ of simplexes of $K^{(n+1)}$ with $s(\sigma_k)\rightarrow\infty$ we have $\lim_{k\mapsto\infty}\rm{diam}(f(\sigma_k))=0$.
\end{itemize}
\end{pro}

\begin{proof}
Suppose $M$ is $LC^n$ and embed $(M,d)$ isometrically in a metric $ANR$-space $(P,\rho)$ as a closed subset. According to the proof of Theorem 2.1, there are metrizable space $Y_P$ and two maps $p\colon Y_P\to P$ and $q\colon Y_p\to M$ such that $\dim Y_P\leq n+1$, $p$ is $(n+1)$-invertible and $q$ extends the map
$p|p^{-1}(M)$.
Using the proof of Nhu's theorem \cite[Theorem 1.1]{nhu} for $ANR$'s, we can find a zero-sequence $\mathcal V=(\mathcal V_k)_{k\in\mathbb N}$ of $P$ such that 
any map $h_0\colon K^{(0)}\to P$ with $h_0(V)\in V$ for each $V\in K^{(0)}$, where $K$ is a subcomplex of $\rm{Tel}(\mathcal U)$, extends to a map $h\colon|K|\to P$ such that $\lim_{k\mapsto\infty}\rm{diam}(h(\sigma_k))=0$ for any sequence $\{\sigma_k\}$ of simplex of $K$ with $s(\sigma_k)\rightarrow\infty$. Let us show that the sequence  $(\mathcal U_k)_{k\in\mathbb N}$, $\mathcal U_k=\{V\cap M:V\in\mathcal V_k\}$, is as required. Indeed, for each $U=V\cap M\in\mathcal U_k$ define $W(U)=V\backslash(p(q^{-1}(M\backslash U)))$ and consider the open families $\mathcal W_k=\{W(U):U\in\mathcal U_k\}$, $k\in\mathbb N$. We may assume that each $U$ is a proper subset of $M$, so
$W_k\neq\varnothing$ for all $k$. Note that $\mathcal W_k$ may not cover $P$, but any $\mathcal W_k$ covers $M$. Moreover, $\rm{mesh}\mathcal{W}_k\leq\rm{mesh} \mathcal{V}_k$, so $\lim_{k\mapsto\infty}\rm{mesh} \mathcal{W}_k=0$. If $K$ is a subcomplex of $\rm{Tel}(\mathcal U)$, take any map $f_0\colon K^{(0)}\to M$ with $f_0(U)\in U$. Hence, $f_0$ extends to a map $g\colon|K|\to P$ such that $\lim_{k\mapsto\infty}\rm{diam}(g(\sigma_k))=0$ for any sequence $\{\sigma_k\}$ of simplexes of $K$ with $s(\sigma_k)\rightarrow\infty$. Finally, let $f\colon|K^{(n+1)}|\to M$ be the map
$q\circ\widetilde{g}$, where $\widetilde{g}\colon|K^{(n+1)}|\to Y_P$ is a lifting of $g$. To show that $f$ satisfies condition $(*)$, fix an $\epsilon>0$, a cover $\mathcal W_m$ with $\rm{mesh}(\mathcal{W}_m)<\epsilon$ and a sequence $\{\sigma_k\}\subset K^{(n+1)}$ with $\lim_{k\mapsto\infty}s(\sigma_k)=\infty$.
Then $\lim_{k\mapsto\infty}\rm{mesh}(g(\sigma_k))=0$, so there exists $k_0$ such that
$g(\sigma_k)\subset\bigcup\mathcal W_m$ for all $k\geq k_0$ (recall that $g(U)\in M$ for all $U\in K^{(0)}$).
Hence, for any two different points $x,y\in\sigma_k$ with $k\geq k_0$ we have $g(x)\in W(U_x)$ and $g(y)\in W(U_y)$ for some $W(U_x),W(U_y)\in\mathcal W_m$. So, according to the definition of $W(U)$, $f(x)\in U_x$ and $f(y)\in U_y$. Therefore,
$\rho(f(x),g(x))\leq\rm{diam}(W(U_x))<\epsilon$ and, similarly, $\rho(f(y),g(y))<\epsilon$. Consequently, $d(f(x),f(y)<2\cdot\epsilon+\rm{diam}(g(\sigma_k))$.

To prove the other implication, embed $M$ as a closed subset of a metrizable space $Z$ with $\dim Z\backslash M\leq n+1$ and follow the proof of
implication $(iii)\Rightarrow (i)$ from \cite[Theorem 1.1]{nhu} to obtain that $M$ is a retract of a neighborhood $W_1$ of $M$ in $Z$ (the only difference is that in Fact 1.2 from \cite{nhu} we take the cover $\mathcal V$ of $W_1\backslash M$ to be of order $\leq n+1$, so the nerve $N(\mathcal V)$ is a complex of dimension $\leq n+1$). Then by \cite[chapter V, Theorem 3.1]{hu}, $M$ is $LC^n$.
\end{proof}

\section{$UV^n$-maps and $ALC^n$-spaces}

In this section we provide spectral characterizations of non-metrizable $ALC^n$-compacta and cell-like compacta.
Recall that a map $f\colon X\to Y$ between compact spaces is said to be soft \cite{s76} if for every compactum $Z$, its closed subset $A\subset Z$ and maps
$h\colon A\to X$ and $g\colon Z\to Y$ with $f\circ h=g|A$ there exists a lifting $\overline{g}\colon Z\to X$ of $g$ extending $h$.
\begin{thm}\label{alc}
A compactum $X$ is $ALC^n$ if and only if $X$ is the limit space of a $\sigma$-complete inverse system $\displaystyle S=\{X_\alpha, p^{\beta}_\alpha, \alpha<\beta<\tau\}$ consisting of compact metrizable $LC^n$-spaces $X_\alpha$ such that all bonding projections $p^{\beta}_\alpha$, as a well all limit projections $p_\alpha$, are $UV^n$-maps.
\end{thm}
\begin{proof}
Suppose that $X$ is the limit space of an inverse system $\displaystyle S=\{X_\alpha, p^{\beta}_\alpha, \alpha<\beta<\tau\}$ such that each $X_\alpha$ is a metric $LC^n$-compactum and all $p^{\beta}_\alpha$ are $UV^n$-maps. We embed $X$ in a Tychonoff cube $\mathbb I^B$, where $\mathbb I=[0,1]$ and the cardinality of $B$ is equal to $\tau$.  According to Shchepin's spectral theorem \cite{s76}, we can assume that $B$ is the union of countable sets $B_\alpha$, $\alpha\in A$, such that $B_\alpha\subset B_\beta$ for $\alpha<\beta$, $B_\gamma=\bigcup\{B_{\gamma(k)}:k=1,2,..\}$ for any chain
$\gamma(1)<\gamma(2)<..$ with $\gamma=\sup\{\gamma(k):k\geq 1\}$, and each $p^{\beta}_\alpha\colon X_\beta\to X_\alpha$ is the restriction of the projection  $q^{\beta}_\alpha\colon\mathbb I^{B_\beta}\to\mathbb I^{B_\alpha}$. So, each $X_\alpha=q_\alpha(X)$ is a subset of $\mathbb I^{B_\alpha}$, where $q_\alpha$ denotes the projection $q_\alpha\colon\mathbb I^B\to\mathbb I^{B_\alpha}$. We also denote $q_\alpha|X$ by $p_\alpha$.
Choose $x_0\in X$ and its neighborhood $U\subset X$. There exists $\alpha_0<\tau$ and an open set $U_0\subset X_{\alpha_0}$ with
$x_0\in p^{-1}_{\alpha_0}(U_0)\subset\overline{p^{-1}_{\alpha_0}(U_0)}\subset U$. Since $X_{\alpha_0}$ is $LC^n$, there exists a neighborhood $V_0$ of $x_{\alpha_0}=p_{\alpha_0}(x_0)$ such that for all $k\leq n$ the inclusion $j_0:V_0\hookrightarrow U_0$ generates trivial homomorphisms $j_0^*:\pi_k(V_0)\to\pi_k(U_0)$ between the homotopy groups. Let $V=p^{-1}_{\alpha_0}(V_0)$ and $\widetilde{U}$ be an open set in $\mathbb I^B$ extending $U$.
Choose a finite family $\omega=\{W_1,..,W_m\}$ of open sets from the ordinary base of $\mathbb I^B$ such that $W=\bigcup_{i=1}^{i=m}W_i$ covers $\overline{p^{-1}_{\alpha_0}(U_0)}$ and $W\subset\widetilde{U}$. Then we can find $\alpha_1>\alpha_0$ with $q^{-1}_{\alpha_1}(q_{\alpha_1}(W_i))=W_i$, $i\leq m$.
Denote $V_1=(p^{\alpha_1}_{\alpha_0})^{-1}(V_0)$ and $U_1=(p^{\alpha_1}_{\alpha_0})^{-1}(U_0)$. Obviously, $q_{\alpha_1}(W)$ is open in $\mathbb I^{B_{\alpha_1}}$ containing $U_1$. Take an open in $\mathbb I^{B_{\alpha_1}}$ extension $\widetilde{V_1}$ of $V_1$ with $\widetilde{V_1}\subset q_{\alpha_1}(W)$. Since $\widetilde{V_1}$ is an $ANR$ and $V_1$ is $LC^n$ (as an open subset of $X_{\alpha_1}$), by Proposition 2.3, there exists an open in $\mathbb I^{B_{\alpha_1}}$ extension $G$ of $V_1$ which is contained in $\widetilde{V_1}$ with the following property: for every map $h\colon Z\to G$, where $Z$ is at most $n$-dimensional metric space, there exists a map $h_1\colon Z\to V_1$ such that $h$ and $h_1$ are homotopic in $\widetilde{V_1}$.
Finally, let $\widetilde{V}=q^{-1}_{\alpha_1}(G)$. It is easily seen that $\widetilde{V}$ is an open extension of $V$ and $\widetilde{V}\subset W\subset\widetilde{U}$. Consider a map $f\colon\mathbb S^k\to\widetilde{V}$, where $k\leq n$. Then there exists a map $g\colon\mathbb S^k\to V_1$ such that $q_{\alpha_1}\circ f$ and $g$ are homotopic in $\widetilde{V_1}$. We are going to show that $g$ is homotopic to a constant map in the set $U_1$. This will be done if the inclusion $j_1\colon V_1\hookrightarrow U_1$ generates a trivial homomorphism $j_1^*:\pi_k(V_1)\to\pi_k(U_1)$.
To this end, we consider the following commutative diagram:
{ $$
\begin{CD}
\pi_k(V_1)@>{{j_1^*}}>>\pi_k(U_1)\\
@ VV{(p^{\alpha_1}_{\alpha_0})^*}V
@VV{(p^{\alpha_1}_{\alpha_0})^*}V\\
\pi_k(V_0)@>{{j_0^*}}>>\pi_k(U_0)
\end{CD}
$$}\\

Since $X_{\alpha_1}$ is $LC^n$ and the map $p^{\alpha_1}_{\alpha_0}$ is $UV^n$, each fiber of $p^{\alpha_1}_{\alpha_0}$ is an $UV^n$-set in 
$X_{\alpha_1}$, see Proposition 2.5. Then, according to \cite[Theorem 5.3]{du}, both vertical homomorphisms from the above diagram are isomorphisms.
This implies that $j_1^*$ is trivial because so is $j_0^*$. Hence, $q_{\alpha_1}\circ f$ is homotopic to a constant map in the set $\widetilde{V_1}\cup U_1\subset q_{\alpha_1}(W)$ (recall that $q_{\alpha_1}\circ f$ is homotopic to $g$ in $\widetilde{V_1}$ and $g$ is homotopic to a constant map in $U_1$). Therefore, there is a map $f_1\colon\mathbb B^{n+1}\to q_{\alpha_1}(W)$ extending $q_{\alpha_1}\circ f$. Since $q_{\alpha_1}$ is a soft map, and $W=q^{-1}_{\alpha_1}(q_{\alpha_1}(W))$,  $f$ can be extended to a map $\tilde{f}\colon\mathbb B^{n+1}\to W$. Thus, the inclusion $\widetilde{V}\hookrightarrow\widetilde{U}$ generates a trivial homomorphism between $\pi_k(\widetilde{V})$ and
$\pi_k(\widetilde{U})$ for any $k\leq n$. So, $X$ is $ALC^n$.

Now, suppose $X$ is $ALC^n$, and consider $X$ as a subset of some $\mathbb I^B$. Since the sets $V$ and $\widetilde{V}$ in the $ALC^n$ definition depend on the point $x$, the set $U$ and its open extension $\widetilde{U}$, respectively, we use the notations $\lambda(x,U)=V$ and $\lambda(x,U,\widetilde{U})=\widetilde{V}$. First, we show that the sets $V$ and $\widetilde{V}$  can be chosen to be functionally open in $X$ and in $\mathbb I^B$, respectively. Indeed, if $U\subset X$ is a neighborhood of $x\in X$,  we take a functionally open in $X$ neighborhood $V^*$ of $x$ with $V^*\subset\lambda(x,U)$. Then for a given open in $\mathbb I^B$ extension $\widetilde{U}$ of $U$ and every $y\in V^*$ choose a functionally open in $\mathbb I^B$ neighborhood $G(y)$ of $y$ with  $G(y)\subset\lambda(x,U,\widetilde{U})\cap G$, where $G$ is an open in $\mathbb I^B$ extension of $V^*$. Since $V^*$, as a functionally open subset of $X$, is Lindel\"{o}f, there exist  countably many sets $G(y_i)$ whose union covers $V^*$. Obviously the set $\widetilde{G}=\bigcup_{i=1}^{\infty}W(y_i)$ is a functionally open in $\mathbb I^B$ extension of $V^*$ which is contained in $\lambda(x,U,\widetilde{U})$. So, every map from $\mathbb S^k$ to $\widetilde{G}$, $k\leq n$, is homotopic in $\widetilde{U}$ to a constant map. 

Therefore, for every open set $U\subset X$ and every $x\in U$ there exists a functionally open in $X$ set $V=\lambda(x,U)$ such that for any open in $\mathbb I^B$ extension $\widetilde{U}$ of $U$ we can find a functionally open in $\mathbb I^B$ extension $\lambda(x,U,\widetilde{U})$  of $V$  contained in $\widetilde{U}$ with all homomorphisms $\pi_k(\lambda(x,U,\widetilde{U})\to\pi_k(\widetilde{U})$, $k\leq n$, being trivial.
If $U$ is functionally open in $X$, then it is Lindel\"{o}f and there are countably many $x_i\in U$ such that $\{\lambda(x_i,U):i=1,2,..\}$ is a cover of $U$. We fix such a countable cover $\gamma(U)$ for any functionally open set $U\subset X$.

Let $A\subset B$ and $W_0,W_1,..,W_k$ be elements of the standard open base $\mathcal B_A$ for $\mathbb I^A$ such that
\begin{itemize}
\item[$(w)$] $\varnothing\neq\overline{W_0}\cap X_A\subset\bigcup_{i=1}^{i=k}W_i$.
\end{itemize}
Here, $X_A=q_A(X)$, where $q_A:\mathbb I^B\to\mathbb I^A$ is the projection.
We also denote by $[W_0,W_1,..,W_k]_A$ the set $q_A^{-1}\big((\bigcup_{i=1}^{i=k}W_i)\backslash(X_A\backslash W_0)\big)$. Observe that
$[W_0,W_1,..,W_k]_A\cap X=q_A^{-1}(W_0)\cap X$, so $[W_0,W_1,..,W_k]_A\cap X$ is functionally open in $X$.
Moreover, if $\gamma(q_A^{-1}(W_0)\cap X)=\{V(y_i):i=1,2,..\}$, where $y_i\in q_A^{-1}(W_0)\cap X$ for all $i$, we consider the sets
$\widetilde{V(y_i)}=\lambda(y_i,q_A^{-1}(W_0)\cap X,[W_0,W_1,..,W_k]_A)$. So, each $\widetilde{V(y_i)}$ is a functionally open extension of $V(y_i)$ and all homomorphisms $\pi_k(\widetilde{V(y_i)})\to\pi_k([W_0,W_1,..,W_k]_A)$, $k\leq n$, are trivial. Denote the family
$\{\widetilde{V(y_i)}:i=1,2,..\}$ by $\widetilde{\gamma}([W_0,W_1,..,W_k]_A)$.
We say that a set $A\subset B$ is admissible if the following holds:
\begin{itemize}
\item $q_A^{-1}(q_A(V))=V$ for all $V\in\widetilde{\gamma}([W_0,W_1,..,W_k]_A)$ and all finitely many  elements $W_0,W_1,..,W_k$ of $\mathcal B_A$ satisfying condition $(w)$;
\item $p_A^{-1}(p_A(V))=V$ for all $V\in\gamma(q_A^{-1}(W)\cap X)$ and $W\in\mathcal B_A$, where $p_A\colon X\to X_A$ is the restriction $q_A|X$.
\end{itemize}
Recall that for any functionally open set $U$ in $\mathbb I^B$ (resp., in $X$) there is a countable set $s(U)\subset B$ such that $q_{s(U)}^{-1}(q_{s(U)}(U))=U$  (resp., $p_{s(U)}^{-1}(p_{s(U)}(U))=U$).

\smallskip
\textit{Claim $2$. For any set $A\subset B$ there exists an admissible set $C\subset B$ of cardinality $|A|.\aleph_0$ containing $A$}.

\smallskip
 We construct by induction sets $A=A_0\subset A_1\subset...\subset A_k\subset A_{k+1}\subset..$ of cardinality $|A|.\aleph_0$ such that:
\begin{itemize}
\item $s(V)\subset A_{k+1}$ for all $V\in\gamma(q_{A_k}^{-1}(W)\cap X)$ and all $W\in\mathcal B_{A_k}$.
\item $s(V)\subset A_{k+1}$ for all $V\in\widetilde{\gamma}([W_0,W_1,..,W_m]_{A_k})$ and all finitely many $W_0,W_1,..,W_m\in\mathcal B_{A_k}$ satisfying condition $(w)$.
\end{itemize}
The construction follows from the fact that the cardinality of each base $\mathcal B_{A_k}$ is $|A|.\aleph_0$ and the families $\gamma(q_{A_k}^{-1}(W)\cap X)$
and $\widetilde{\gamma}([W_0,W_1,..,W_m]_{A_k})$ are countable provided $W, W_0,..,W_k\in\mathcal B_{A_k}$. It is easily seen that the set $C=\bigcup_{k=1}^{\infty}A_k$ is as required.

\smallskip
\textit{Claim $3$. $X_A$ is an $ALC^n$-space for every admissible set $A\subset B$}.

\smallskip
Let $y\in U$, where $U$ is open in $X_A$. Take $x\in X$ and $W_0\in\mathcal B_{A}$ containing $y$  such that $\overline{W_0}\cap X_A\subset U$ and $q_A(x)=y$.
Then $x$ belongs to some $V_x\in\gamma(q_A^{-1}(W_0)\cap X)$. Since $s(V_x)\subset A$ (recall that $A$ is admissible), $p_{A}^{-1}(p_{A}(V_x))=V_x$.
So, $V=p_A(V_x)$ is a functionally open in $X_A$ neighborhood of $y$, which is contained in $U$. Take any open extension $\widetilde{U}\subset\mathbb I^A$ of $U$, and finitely many $W_1,..,W_k\in\mathcal B_A$ satisfying $\overline{W_0}\cap X_A\subset\bigcup_{i=1}^{i=k}W_i\subset\widetilde{U}$.
Since $V_x\in\gamma(q_A^{-1}(W_0)\cap X)$, the set $\widetilde{V_x}=\lambda(x,q_A^{-1}(W_0)\cap X,[W_0,W_1,..,W_k]_A)$ is a functionally open extension of $V_x$  with $\widetilde{V_x}\in\widetilde{\gamma}([W_0,W_1,..,W_k]_{A})$. Then $s(\widetilde{V_x})\subset A$ and $\widetilde{V}=q_A(\widetilde{V_x})$ is a functionally open in $\mathbb I^A$ extension of $V$. We are going to show that all homomorphisms $\pi_k(\widetilde{V})\to\pi_k(\widetilde{U})$, $k\leq n$, are trivial. Indeed, every map $f\colon\mathbb S^k\to\widetilde{V}$ can be lifted to a map $g\colon\mathbb S^k\to\widetilde{V_x}$ because $q_A^{-1}(\widetilde{V})=\widetilde{V_x}$ and $q_A$ is a soft map. Recall that $\widetilde{V_x}$ belongs to $\widetilde{\gamma}([W_0,W_1,..,W_k]_{A})$, so $g$ can be extended to a map $\widetilde{g}:\mathbb B^{k+1}\to [W_0,W_1,..,W_k]_{A}$. Finally, $q_A\circ\widetilde{g}:\mathbb B^{k+1}\to q_A([W_0,W_1,..,W_k]_{A})\subset\widetilde{U}$ is an extension of $f$. This completes the proof of Claim 3.

\smallskip
\textit{Claim $4$. Let $A_2\subset A_1$ be admissible subsets of $B$. Then each fiber of $p^{A_1}_{A_2}\colon X_{A_1}\to X_{A_2}$ is $UV^n$}.

\smallskip
Let $x\in X_{A_2}$ and $U\subset\mathbb I^{A_1}$ be an open set containing $F=(p^{A_1}_{A_2})^{-1}(x)$. Take $W_0'\in\mathcal B_{A_2}$ with $x\in W_0'$ and
$(p^{A_1}_{A_2})^{-1}(\overline{W_0'}\cap X_{A_2})\subset U$. So, $(q^{A_1}_{A_2})^{-1}(\overline{W_0'})\cap X_{A_1}\subset U$. Next, choose
$W_1,..,W_k\in\mathcal B_{A_1}$ such that $(q^{A_1}_{A_2})^{-1}(\overline{W_0'})\cap X_{A_1}\subset\bigcup_{i=1}^{i=k}W_i\subset U$. Obviously,
$W_0=(q^{A_1}_{A_2})^{-1}(W_0')\in\mathcal B_{A_1}$ and $\overline{W_0}\cap X_{A_1}\subset\bigcup_{i=1}^{i=k}W_i$. Let $y\in X$ with $q_{A_2}(y)=x$. Then $y$ belongs to some $V_y\in\gamma(q_{A_1}^{-1}(W_0)\cap X)$. Because $q_{A_1}^{-1}(W_0)=q_{A_2}^{-1}(W_0')$ and $A_2$ is admissible, we have
$s(V_y)\subset A_2$. Hence, $p_{A_2}^{-1}(p_{A_2}(V_y))=V_y$ and $p_{A_2}^{-1}(x)=p_{A_1}^{-1}(F)\subset V_y$. Then $\widetilde{V_y}=\lambda(y,q_A^{-1}(W_0)\cap X,[W_0,W_1,..,W_k]_{A_1})$ is a functionally open extension of $V_y$ and $\widetilde{V_y}\in\widetilde{\gamma}([W_0,W_1,..,W_k]_{A_1})$. Because $s(\widetilde{V_y})\subset A_1$, $q_{A_1}^{-1}(q_{A_1}(\widetilde{V_y}))=\widetilde{V_y}$ and $V=q_{A_1}(\widetilde{V_y})$ is an open subset of $\mathbb I^{A_1}$ such that
$F\subset V\subset\bigcup_{i=1}^{i=k}W_i\subset U$. Then, as in the proof of Claim 3, we can show that the inclusion $V\hookrightarrow U$ generates
trivial homomorphisms $\pi_k(V)\to\pi_k(U)$. Hence, $F$ is $UV^n$.

\smallskip
\textit{Claim $5$. The union of any increasing sequence of admissible subsets of $B$ is also admissible.}

\smallskip
This claim follows directly from the definition of admissible sets.

Now we can complete the proof of Theorem 3.1. According to Claim 2 and Claim 5, the set $B$ is covered by a family $\mathcal S$ of countable sets such that $\mathcal S$
is stable with respect to countable unions. Then, by Claim 3, each $X_A$, $A\in\mathcal S$, is a metric $ALC^n$-compactum . Hence, Proposition 2.7 yields that all spaces
$X_A$, $A\in\mathcal S$, are $LC^n$. Moreover, the projections $p^{A_1}_{A_2}$ are $UV^n$-maps for any $A_1,A_2\in\mathcal S$ with $A_2\subset A_1$. Because the set $B$ is admissible, it follows from Claim 4 that the limit projections $p_A\colon X\to X_A$, $A\in\mathcal S$, are also $UV^n$-maps. Therefore, $X$ is limit space of the $\sigma$-complete inverse system $\{X_A, p^{A_1}_{A_2},A,A_1,A_2\in\mathcal S\}$.
\end{proof}

\begin{cor}
Any $LC^n$-compactum $X$ is an $ALC^n$-space.
\end{cor}
\begin{proof}
We embed $X$ in some $\mathbb I^B$ and let $A_0\subset B$ be a countable set. According to the factorization theorem of Bogatyi-Smirnov \cite[Theorem 3]{bs}, there is a metric compactum $Y_1$ and maps $g_1\colon X\to Y_1$, $h_1\colon Y_1\to X_{A_0}$ such that $p_{A_0}=h_1\circ g_1$ and all fibers of $g_1$ are $UV^n$-sets in $X$. Then, by \cite[Theorem 5.4]{du}, $Y_1$ is $LC^n$. Since $g_1$ depends on countably many coordinates, there is a countable set $A_1\subset B$ containing $A_0$ and a map $f_1\colon X_{A_1}\to Y_1$ such that  $f_1\circ p_{A_1}=g_1$. In this way we construct countable sets $A_k\subset A_{k+1}\subset B$ and $LC^n$ metric compacta $Y_k$ together with maps $g_k\colon X\to Y_k$, $f_k\colon \colon X_{A_k}\to Y_k$ and $h_k\colon Y_k\to X_{A_{k-1}}$ such that $g_k=f_k\circ p_{A_k}$,  $p_{A_{k-1}}=h_k\circ g_k$ and the fibers of each $g_k$ are $UV^n$-sets in $X$.
Let $A$ be the union of all $A_k$. Then $X_A$ is the limit space of the inverse sequence $\mathcal S=\{Y_k, s^{k+1}_k=f_k\circ h_{k+1}\}$. According to \cite[Theorem 5.3]{du}, for all open sets $U\subset Y_k$ the group $\pi_m(U)$ is isomorphic to $\pi_m(g_k^{-1}(U))$, $m=0,1,..,n$. This property of the maps $g_k$ implies that any $s^{k+1}_k\colon Y_{k+1}\to Y_k$ is an $UV^n$-map. Hence, by Theorem \ref{alc}, $X_A$ is an $ALC^n$-compactum (as the limit of an inverse sequence of metric $LC^n$-compacta and bounding $UV^n$-maps). Finally, by Theorem 2.7, $X_A$ is also an $LC^n$-space. Moreover, for any $y\in X_A$ we have
$p_A^{-1}(y)=\bigcap_{k\geq 1}g_k^{-1}(y_k)$, where $y_k=s_k(y)$ with $s_k\colon X_A\to Y_k$ being the projections of $\mathcal S$. Because all $g_k^{-1}(y_k)$ are $UV^n$-sets in $X$, so is the set $p_A^{-1}(y)$.
Therefore, every countable subset $A_0$ of $B$ is contained in an element of the family $\mathcal A$ consisting of all countable sets $A\subset B$ such that $X_A$ is $LC^n$ and the fibers of the map $p_A$ are $UV^n$-sets in $X$. It is easily seen that the union of an increasing sequence of elements of $\mathcal A$ is again from $\mathcal A$, and that $p^C_A\colon X_C\to X_A$ is an $UV^n$-map for all $A,C\in\mathcal A$ with $A\subset C$. So, the inverse system $\{X_A, p^C_A, A,C\in\mathcal A\}$ is $\sigma$-complete and consists of metric $LC^n$-compacta and $UV^n$-bounding maps. Then by Theorem \ref{alc}, $X$ is $ALC^n$ (observe that the proof of the "if" part of Theorem \ref{alc} does not need the assumption that all limit projections are $UV^n$-maps).
\end{proof}

Theorem \ref{alc} shows that the class of $ALC^n$-compacta is adequate to the class of $UV^n$-maps. Next theorem provides another classes of compacta adequate to all continuous maps.

\begin{thm}\label{uv}
A compactum $X$ is a cell-like $($resp., $UV^n$$)$ space if and only if $X$ is the limit space of a $\sigma$-complete inverse system consisting of
cell-like $($resp., $UV^n$$)$ metric compacta.
\end{thm}

\begin{proof}
 Suppose $X$ is a cell-like compactum. Because this means that $X$ has a shape of a point, we can apply  Corollary 8.4.8 from \cite{ch1} stating that if $\varphi$ is a shape isomorphism between the limit spaces of two $\sigma$-complete inverse systems $\{X_\alpha, p^{\beta}_\alpha, \alpha\in A\}$ and  $\{Y_\alpha, q^{\beta}_\alpha, \alpha\in A\}$ of metric compacta, then the set of those $\alpha\in A$ for which there exist shape isomorphisms $\varphi_\alpha: X_\alpha\to Y_\alpha$ satisfying
$Sh(q_\alpha)\circ\varphi=\varphi_\alpha\circ Sh(p_\alpha)$ is cofinal and closed in $A$. So, according to this corollary, $X$ is the limit space of a $\sigma$-complete inverse system consisting of metric cell-like compacta. In case $X$ is an $UV^n$-compactum, it has an $n$-shape of a point (this notion was introduced by Chigogidze in \cite{ch2}), and the above arguments apply.

Suppose now that $X$ is the limit space of a $\sigma$-complete inverse system $\{X_\alpha, p^{\beta}_\alpha, \alpha\in A\}$ such that all $X_\alpha$ are metric cell-like compacta. As in the proof of Theorem \ref{alc}, we can embed $X$ in a Tychonoff cube $\mathbb I^B$, where $B$ is the union of countable sets $B_\alpha$, $\alpha\in A$, such that $B_\alpha\subset B_\beta$ for $\alpha<\beta$, $B_\gamma=\bigcup\{B_{\gamma(k)}:k=1,2,..\}$ for any chain
$\gamma(1)<\gamma(2)<..$ with $\gamma=\sup\{\gamma(k):k\geq 1\}$, and each $p^{\beta}_\alpha\colon X_\beta\to X_\alpha$ is the restriction of the projection  $q^{\beta}_\alpha\colon\mathbb I^{B_\beta}\to\mathbb I^{B_\alpha}$. Then $X_\alpha=q_\alpha(X)\subset\mathbb I^{B_\alpha}$ with  $q_\alpha$ being the projection from $\mathbb I^B$ onto $\mathbb I^{B_\alpha}$. If $U$ is a neighborhood of $X$ in $\mathbb I^B$, there is $\alpha$ and an open set $U_\alpha$ in
$\mathbb I^{B_\alpha}$ such that $q_\alpha^{-1}(U_\alpha)\subset U$. Since $X_\alpha$ is a cell-like space, there exists a closed neighborhood
$V_\alpha\subset\mathbb I^{B_\alpha}$ of $X_\alpha$ contractible in $U_\alpha$. Using that $q_\alpha$ is a soft map, we conclude that $q_\alpha^{-1}(V_\alpha)$ is contractible in $q_\alpha^{-1}(U_\alpha)$. Similarly, we can show that any limit space of a $\sigma$-complete inverse system of metric $UV^n$-compacta is also an $UV^n$-compactum.
\end{proof}

\end{document}